\documentclass[12pt]{amsart}

\usepackage{amsmath,amssymb}
\usepackage{bm}
\usepackage{graphicx}
\usepackage{ascmac}
\usepackage{amsthm}
\usepackage{amsfonts}
\usepackage{mathrsfs}
\usepackage{mathtools}
\usepackage[all]{xy}
\mathtoolsset{showonlyrefs=true}

\theoremstyle{definition}
\newtheorem{thm}{Theorem}[section]
\newtheorem{dfn}[thm]{Definition}

\newtheorem{cor}[thm]{Corollary}
\newtheorem{prop}[thm]{Proposition}
\newtheorem{lem}[thm]{Lemma}
\newtheorem{rem}[thm]{Remark}
\newtheorem*{nota}{Notation}
\newtheorem{axm}[thm]{Axiom}

\newtheorem*{oftp}{Organization of this paper}
\newtheorem*{ack}{Acknowledgment}

\numberwithin{thm}{section}

\newcommand{\Zpn}{\mathbb{Z}_{>0}}

\newcommand{\Z}{\mathbb{Z}}
\newcommand{\Zp}{{\mathbb{Z}}_p}

\newcommand{\Zpb}{\breve{\Z}_p}
\newcommand{\Q}{\mathbb{Q}}
\newcommand{\Qbar}{\overline{\Q}}
\newcommand{\Qp}{{\Q}_p}
\newcommand{\Qpt}{{\Q}_{p}^{\times}}

\newcommand{\Qpb}{\breve{\Q}_p}
\newcommand{\Qpbt}{\breve{\Q}_p^{\times}}
\newcommand{\Qpbar}{\overline{\Q}_p}

\newcommand{\Ql}{{\mathbb{Q}}_{\ell}}
\newcommand{\Qlbar}{\Qbar_{\ell}}
\renewcommand{\O}{\mathcal{O}}

\newcommand{\R}{\mathbb{R}}
\newcommand{\C}{\mathbb{C}}
\newcommand{\F}{\mathbb{F}}
\newcommand{\Fp}{\mathbb{F}_p}

\newcommand{\Fpbar}{\overline{\mathbb{F}}_p}
\newcommand{\A}{\mathbb{A}}
\newcommand{\Sbar}{\overline{S}}
\newcommand{\D}{\mathbb{D}}
\newcommand{\G}{\mathbb{G}}
\newcommand{\bS}{\mathbb{S}}

\newcommand{\bG}{\mathbf{G}}

\newcommand{\bI}{\mathbf{I}}
\newcommand{\sI}{\mathscr{I}}
\newcommand{\sS}{\mathscr{S}}
\newcommand{\sShat}{\widehat{\sS}}
\newcommand{\sSbar}{\overline{\sS}}

\newcommand{\X}{\mathbb{X}}
\newcommand{\bX}{\mathbf{X}}
\newcommand{\sX}{\mathscr{X}}

\newcommand{\bT}{\mathbf{T}}
\newcommand{\bZ}{\mathbf{Z}}

\renewcommand{\hbar}{\overline{h}}

\newcommand{\E}{\mathbf{E}}

\newcommand{\bc}{\mathbf{c}}
\newcommand{\bi}{\mathbf{i}}

\newcommand{\bx}{\mathbf{x}}
\newcommand{\cA}{\mathcal{A}}
\newcommand{\cB}{\mathcal{B}}
\newcommand{\cG}{\mathcal{G}}

\newcommand{\cN}{\mathcal{N}}

\newcommand{\cV}{\mathcal{V}}

\newcommand{\Eb}{\breve{E}}
\newcommand{\Kb}{\breve{K}}

\DeclareMathOperator{\Ker}{Ker}

\DeclareMathOperator{\Hom}{Hom}
\DeclareMathOperator{\End}{End}
\DeclareMathOperator{\id}{id}
\DeclareMathOperator{\Aut}{Aut}
\DeclareMathOperator{\Lie}{Lie}

\DeclareMathOperator{\GL}{GL}
\DeclareMathOperator{\GSp}{GSp}

\DeclareMathOperator{\nilp}{Nilp}

\DeclareMathOperator{\spec}{Spec}
\DeclareMathOperator{\spf}{Spf}
\DeclareMathOperator{\dR}{dR}
\DeclareMathOperator{\loc}{loc}

\DeclareMathOperator{\red}{red}

\DeclareMathOperator{\Res}{Res}

\DeclareMathOperator{\diag}{diag}

\DeclareMathOperator{\RZ}{RZ}

\DeclareMathOperator{\Sh}{Sh}

\DeclareMathOperator{\der}{der}
\DeclareMathOperator{\bs}{basic}

\DeclareMathOperator{\Adm}{Adm}
\DeclareMathOperator{\Fil}{Fil}

\DeclareMathOperator{\Ad}{Ad}
\DeclareMathOperator{\Def}{Def}
\DeclareMathOperator{\rR}{R}
\DeclareMathOperator{\rH}{H}
\DeclareMathOperator{\et}{\acute{e}t}

\usepackage{geometry}
\title[Notes on Rapoport--Zink spaces]{Notes on Rapoport--Zink spaces of Hodge type with parahoric level structure}
\author[Y.Oki]{Yasuhiro Oki}
\address{Graduate School of Mathematical Sciences, 
the University of Tokyo, 3-8-1 Komaba, Meguro-ku, Tokyo 153-8914, Japan. }
\email{oki@ms.u-tokyo.ac.jp}

\begin{document}
\maketitle
\begin{abstract}
In this article, we treat two questions on Rapoport--Zink spaces of Hodge type constructed by Hamacher and Kim. One of which is their singularities, and the other is $p$-adic uniformization of Shimura varieties. More precisely, we prove that the singularity of a Rapoport--Zink space is controlled by its asssociated local model, and the basic locus of a Kisin--Pappas integral model of a Shimura variety is uniformized by the corresponding Rapoport--Zink space. These results extend the known facts by Rapoport and Zink in the case of PEL type. 
\end{abstract}

\tableofcontents

\section{Introduction}\label{intr}

Rapoport--Zink spaces are one of the most important objects in the theory of Shimura varieties. In the case of PEL type, these were introduced by Rapoport and Zink in \cite{Rapoport1996b}. They also proved some fundamental questions on Rapoport--Zink spaces at the same time. We explain two of them more precisely in the following. 

The first question is the singularities of Rapoport--Zink spaces. In \cite[Chapter 3]{Rapoport1996b}, they defined local models, and proved that the singularity of a Rapoport--Zink space is controlled by a local model. We call such a result as a local model diagram. The second question is $p$-adic uniformization of Shimura varieties. In \cite[Chapter 6]{Rapoport1996b}, they considered a $\Zp$-scheme whose generic fiber is a (finite disjoint union of) Shimura variety, and proved that the completion of the considering $\Zp$-scheme along a basic Newton stratum isomorphic to a finite disjoint union of some quotients of the corresponding Rapoport--Zink space. 

In the case of Hodge type, the above questions were solved in the hyperspecial case. More precisely, Kisin defined and studied integral models of Shimura varieties, and proved that they are smooth (\cite{Kisin2010}). Furthermore, Kim (\cite{Kim2018a}, \cite{Kim2018b}), Howard and Pappas (\cite{Howard2017}) gave constructions of the Rapoport--Zink spaces that are formally smooth, and proved the $p$-adic uniformization theorem for basic loci. 

In this article, we consider the singularities of Rapoport--Zink spaces and the $p$-adic uniformization of Shimura varieties in the case of Hodge type with Bruhat--Tits level structure (we will explain this terminology in later) at $p$. In the sequel of this paper, let $p>2$ be an odd prime number. Let $\Qp$ the field of $p$-adic numbers, and fix an isomorphism $\Qpbar \cong \C$. On the other hand, for a discrete valuation field $F$, we write $O_F$ and $\breve{F}$ for the integer ring of $F$ and the completion of the maximal unramified extension of $F$ respectively. 

Let $(\bG,\bX)$ be a Shimura datum of Hodge type, and denote by $\E$ its reflex field. Here we assume the following on the group $G:=\bG \otimes_{\Q}\Qp$: 
\begin{itemize}
\item $G$ splits over a tamely ramified extension of $\Qp$ and $p\nmid \#\pi_{1}(G^{\der})$,
\end{itemize}
Let $K^p$ be a sufficiently small compact open subgroup of $\bG(\A_f^p)$, where $\A_f^p$ is the finite ad{\`e}le ring of $\Q$ without $p$-component, and $K_p$ the \emph{full} stabilizer of a vertex $\bx$ of the (extended) Bruhat--Tits building of $G$. We call such $K_p$ as a \emph{Bruhat--Tits subgroup}. Put $K:=K^pK_p$. For a place $\mathfrak{p}$ of $\E$ lying above $p$, consider a Kisin--Pappas integral model $\sS_{K}$ of the Shimura variety $\Sh_{K}(\bG,\bX)$ of level $K$ constructed in \cite{Kisin2018}. It is a flat $O_{\E,(\mathfrak{p})}$-scheme, where $O_{\E,(\mathfrak{p})}$ is the valuation subring of $\E$ with respect to $\mathfrak{p}$. By construction, there is $n\in \Zpn$ and a finite morphism from $\sS_{K}$ to the moduli space of prime-to-$p$ polarized abelian varieties of dimension $n$ with certain level structure. This morphism gives us an abelian scheme over $\sS_{K}$, and hence a Newton stratification of the geometric special fiber $\sSbar_{K}$ of $\sS_{K}$: 
\begin{equation*}
\sSbar_{K}=\coprod_{[b]\in B(G,\{\mu\})}\sS_{K,[b]}. 
\end{equation*}
Here $\{\mu\}$ is the $G(\Qpbar)$-conjugacy classes of the inverse of a cocharcter of $G$ defined by $\bX$, and $B(G,\{\mu\})$ is the finite subset of the $\sigma$-conjugacy classes of $G(\Qpb)$ which will be introduced in Section \ref{gisc}. 

On the other hand, for a suitable representative $b$ of $[b]$, Hamacher and Kim constructed in \cite{Hamacher2019} a closed formal subscheme $\RZ_{K_p}(G,\mu,b)$ of the deformation space of the principally polarized $p$-divisible group attached to $z$, under a technical hypothesis. It is a formal scheme over $\spf O_{\Eb}$, where $E$ is the completion of $\E$ at $\mathfrak{p}$. We call $\RZ_{K_p}(G,\mu,b)$ as a \emph{Rapoport--Zink space} for $(G,\mu,b)$ with level $K_p$. It is equipped with a left action of the group $J_{b}(\Qp)$, where $J_{b}(\Qp)$ is the set of $\Qp$-valued points of the centralizer $J_{b}$ of $b\sigma$. Moreover, by construction, we have a \emph{$p$-adic uniformization map}
\begin{equation*}
\RZ_{K_p}(G,\mu,b)\rightarrow \sS_{K},
\end{equation*}
which depends on the choice of $z\in \sS_{K,[b]}(\Fpbar)$. See Section \ref{rzsp} for more details. 

\subsection{Main theorem I: Singularities of Rapoport--Zink spaces}\label{mtlr}

For a triple $(G,\mu,\bx)$ as above, Pappas and Zhu introduced in \cite{Pappas2013} the \emph{local model} $M^{\loc}(G,\mu,\bx)$, which is a flat projective $O_{E}$-scheme. We will give a more precise description in Section \ref{lcmd}. We write $\widehat{M}^{\loc}(G,\mu,\bx)$ for the $p$-adic completion of $M^{\loc}(G,\mu,\bx)\otimes_{O_E}O_{\Eb}$. 

\begin{thm}\label{mtlm}(Theorem \ref{llmd})
\emph{There is a diagram
\begin{equation*}
\xymatrix{
&\widetilde{\RZ}_{K_p}(G,\mu,b)\ar[ld]_{\pi_{b}} \ar[rd]^{q_{b}} & \\
\RZ_{K_p}(G,\mu,b)& &\widehat{M}^{\loc}(G,\mu,\bx),}
\end{equation*}
where $\varphi$ is a $(\cG_{\bx}\otimes_{\Zp}O_E)$-torsor, and $\psi$ is $(\cG_{\bx}\otimes_{\Zp}O_E)$-equivariant and formally smooth of relative dimension $\dim(G)$. }
\end{thm}


We give a strategy of the proof of Theorem \ref{mtlm}. By \cite[Theorem 4.2.7]{Kisin2018}, there is a local model diagram on Kisin--Pappas integral models, which asserts that the singularity of $\sS_{K}$ is controlled by $M^{\loc}(G,\mu,\bx)$. Then we prove that the pull-back of such a diagram by the $p$-adic uniformization map, which is already obtained by the construction of $\RZ_{K_p}(G,\mu,b)$, gives a desired one. 


\subsection{Main theorem I\hspace{-0.5mm}I: $p$-adic uniformization}\label{mtgr}

Keep the notations as above, and let $[b_0]\in B(G,\{\mu\})$ be the unique basic element in the sense of \cite{Kottwitz1985}. Then the stratum $\sS_{K,[b_0]}$ is a non-empty closed subset of $\sS_{K}$ by \cite[Theorem 1.3.13 (2)]{Kisin}. Moreover, we can consider the Rapoport--Zink space $\RZ_{K_p}(G,\mu,b_0)$ for a suitable representation $b_0$ of $[b_0]$. We denote by $(\sShat_{K,O_{\Eb}})_{/\sSbar_{K,[b_0]}}$ the formal completion of $\sS_{K,O_{\Eb}}$ along $\sSbar_{K,[b_0]}$. 

\begin{thm}\label{math}(Theorem \ref{paut})
\emph{For $z\in \sSbar_{K,[b_0]}$, there is an isomorphism of $O_{\Eb}$-formal schemes
\begin{equation*}
\bI_{z}(\Q)\backslash \RZ_{K_p}(G,\mu,b_0)\times \bG(\A_{f}^{p})/K^p\cong (\sShat_{K,O_{\Eb}})_{/\sSbar_{K,[b_0]}}. 
\end{equation*}
Here $\bI_{z}$ is the algebraic group over $\Q$ which is anisotropic modulo center satisfying
\begin{equation*}
\bI_{z}\otimes_{\Q}\Ql \cong
\begin{cases}
\bG \otimes_{\Q}\Ql &\text{if }\ell \neq p,\\
J_{b_0}&\text{if }\ell=p. 
\end{cases}
\end{equation*}}
\end{thm}

We sketch the proof of Theorem \ref{math}. It is based on the argument of \cite{Kim2018b} and \cite{Kisin2017}. We construct a morphism by a $p$-adic uniformization map and the prime-to-$p$ Hecke action on $\sS_{K_p}:=\varprojlim_{K'^p}\sS_{K'^pK_p}$. The most non-trivial assertion is the surjectivity of the constructed map. The main innovation in our proof is an existence of a special point lifting in an isogeny class, which extends the result of \cite{Zhou2020} in the basic case. By using this, we follow the proof of \cite[Proposition (4.4.13)]{Kisin2017} to prove the desired surjectivity. However, we need trace it much more carefully. Indeed, the proof of \cite[Proposition (4.4.13)]{Kisin2017} uses the transitivity of the action of $\bG(\A_f^p)$ on $\pi_0(\Sh_{K_p})$, which may failure in general. 

\begin{rem}
We use Theorems \ref{mtlr} and \ref{math} to study the structure of the Rapoport--Zink spaces and the basic loci of Shimura varieties for spinor similitude groups with special maximal parahoric level structure (\cite{Oki2020}). It contains the cases in which the groups are not quasi-split at $p$. 
\end{rem}

\begin{oftp}
In Section \ref{pddf}, we study the set of the $\sigma$-conjugacy classes of $G(\Qpb)$, and introduce the notions of local models and deformations of $p$-divisible groups. In Sections \ref{kpim} and \ref{rzsp}, we recall the definitions of Kisin--Pappas integral models and Rapoport--Zink spaces of Hodge type respectively. In Sections \ref{lmdr} and \ref{padc}, we give proofs of Theorems \ref{mtlm} and \ref{math} respectively. 
\end{oftp}

\begin{ack}
I would like to thank my advisor Yoichi Mieda for his constant support and encouragement. I also thank Wansu Kim and Pol van Hoften for many valuable comments. 

This work was carried out with the support from the Program for Leading Graduate Schools, MEXT, Japan. This work was also supported by the JSPS Research Fellowship for Young Scientists and KAKENHI Grant Number 19J21728.
\end{ack}

\begin{nota}
\begin{itemize}
\item \textbf{Total tensor algebras. }Let $R$ be a commutative ring, and $M$ an $R$-module. We write $M^{*}$ for the dual module of $M$. Moreover, denote by $M^{\otimes}$ the disjoint union of $R$-modules that can be formed from $M$ using the finite operations of taking duals, tensor products, symmetric products and alternating products. We also use these notations for the same objects for modules over sheaves. 
\item \textbf{Symplectic similitude group. } For $n\in \Zpn$ and a commutative ring $R$, let $(V_{2n,R},\psi_{2n})$ be the symplectic space over $R$ whose Gram matrix with respect to the standard basis is
\begin{equation*}
{\begin{pmatrix}
0& J_n\\
-J_n &0
\end{pmatrix}}, 
\end{equation*}
where $J_n$ is the anti-diagonal matrix of size $n$ that has $1$ as every non-zero entry. We define an algebraic group $\GSp_{2n,R}$ over $R$ as the subgroup of $\GL(V_{2n,R})$ which respects $\psi_{2n}$ up to an invertible scalar. 
\item \textbf{Fundamental groups. } For a reductive connected group $G$ over a field $k$, put $X_{*}(G):=\Hom_{\overline{k}}(\G_{m,\overline{k}},G\otimes_{k}\overline{k})$, and write $\pi_1(G)$ for the Borovoi's algebraic fundamental group of $G$. Note that there is a canonical map $X_{*}(G)\rightarrow \pi_1(G)$, which factors through the set of $G(\overline{k})$-conjugacy classes of $X_{*}(G)$. 
\item \textbf{Bruhat--Tits buildings. } Let $G$ be a reductive connected group over $\Qp$. We denote by $\cB(G,\Qp)$ the (extended) Bruhat--Tits building of $G$, and for $\bx\in \cB(G,\Qp)$, let $\cG_{\bx}$ be the smooth group scheme over $\Zp$ associated to $\bx$. Moreover, we write $\cG_{\bx}^{\circ}$ for the identity component of $\cG_{\bx}$. 
\end{itemize}
\end{nota}

\section{$p$-divisible groups and their deformations}\label{pddf}

In this section, we prepare the notions relating with two main results which will be proved in this article. 

\subsection{$\sigma$-conjugacy classes}\label{gisc}

Let $G$ be a connected reductive group over $\Qp$. We define $B(G)$ as the set of $\sigma$-conjugacy classes in $G(\Qpb)$, that is, the set of 
\begin{equation*}
[b]:=\{gb\sigma(g)^{-1}\in G(\Qpb)\mid g\in G(\Qpb)\}
\end{equation*}
for all $b\in G(\Qpb)$. 

First, we recall \cite[Theorem 1.15]{Rapoport1996a}. Fix a maximal torus $T_0$ of $G$, and let $W$ be the absolute Weyl group of $(G,T_0)$. We denote by $\Gamma$ the absolute Galois group of $\Qp$. Put $\cN(G):=(X_{*}(T_0)\otimes_{\Z}\Q/W)^{\Gamma}$. Then implies that there is a unique functorial map
\begin{equation*}
\cN(G)\rightarrow (\pi_1(G)\otimes_{\Z}\Q)^{\Gamma}
\end{equation*}
which gives the identity map when $G$ is a torus. On the other hand, we write $N$ for the composite
\begin{equation*}
\pi_1(G)_{\Gamma}\rightarrow (\pi_1(G)\otimes_{\Z}\Q)_{\Gamma}\rightarrow (\pi_1(G)\otimes_{\Z}\Q)^{\Gamma}, 
\end{equation*}
where the second map is as in \cite[p.162]{Rapoport1996a}, that is, taking avarages over $\Gamma$-orbits. 

On the other hand, there are maps
\begin{equation*}
\overline{\nu}_{G}\colon B(G)\rightarrow \cN(G),\quad \kappa_{G}\colon B(G)\rightarrow \pi_1(G)_{\Gamma},
\end{equation*}
that are functorial with respect to $G$. 

\begin{prop}\label{ntkt}(\cite[Theorem 1.15]{Rapoport1996a})
\emph{The following diagram commutes: 
\begin{equation*}
\xymatrix{
B(G)\ar[r]^{\kappa_G} \ar[d]^{\overline{\nu}_G} &\pi_1(G)_{\Gamma} \ar[d] \\
\cN(G) \ar[r] &(\pi_1(G)\otimes_{\Z}\Q)^{\Gamma},}
\end{equation*}}
\end{prop}

\begin{dfn}
We say that $[b]$ is \emph{basic} if $\overline{\nu}_G([b])$ factors through the center of $G$. 
\end{dfn}

Let $B(G)_{\bs}$ be the set of all basic elements of $B(G)$. We recall some properties on basic $\sigma$-conjugacy classes. For $b\in G(\Qpb)$, let $J_{b}$ be an algebraic group over $\Qp$ satisfying
\begin{equation*}
J_{b}(R)=\{g\in G(R\otimes_{\Qp}\Qpb)\mid gb=b\sigma(g)\}
\end{equation*}
for any $\Qp$-algebra $R$. Then $J_{b}$ is determined by $[b]$ up to isomorphism. 

\begin{prop}\label{bseq}
\emph{For $[b]\in B(G)$, the following are equivalent: 
\begin{enumerate}
\item $[b]$ is basic,
\item $[b]$ lies in the image of the map $B(T)\rightarrow B(G)$ for some (or any) elliptic maximal torus $T$ of $G$,
\item $J_{b}$ is an inner form of $G$ for some (or all) $b\in [b]$. 
\end{enumerate}}
\end{prop}

\begin{proof}
The equivalence between (i) and (ii) is \cite[5.2]{Kottwitz1985}. 
On the other hand, the equivalence between (i) and (iii) is \cite[Proposition 5.3]{Kottwitz1985}. 
\end{proof}

\begin{prop}\label{bskt}(\cite[Theorem 1.15 (i)]{Rapoport1996a})
\emph{The Kottwitz map $\kappa_G$ induces a bijection
\begin{equation*}
B(G)_{\bs}\cong \pi_1(G)_{\Gamma}. 
\end{equation*}}
\end{prop}

For $\mu \in X_{*}(G)$, we denote by $[b_{\bs}(\mu)]\in B(G)_{\bs}$ be the unique element satisfying 
\begin{equation*}
\kappa_{G}([b_{\bs}(\mu)])=\mu^{\diamond},
\end{equation*}
where $\mu^{\diamond}$ is the image of $\mu$ under the canonical map $X_{*}(G)\rightarrow \pi_1(G)$. Let $B(G,\{\mu\})$ be the subset of $B(G)$ defined as \cite[6.2]{Kottwitz1997}. Then \cite[6.4]{Kottwitz1997} asserts that $B(G,\{\mu\})$ contains the unique basic element $[b_{\bs}(\mu)]$. 

\begin{lem}\label{adsp}
\emph{Let $[b]\in B(G,\{\mu\})$ be basic, and $T$ an elliptic maximal torus of $G$. Then there is $g\in G(\Qpbar)$ satisfying the following: 
\begin{enumerate}
\item $\mu':=\Ad(g)\circ \mu \in X_{*}(T)$, 
\item $\overline{\mu}'=\nu_{b}$,
\item the image of $[b_{\bs}(\mu')]$ under $B(T)\rightarrow B(G)$ equals $[b]$. 
\end{enumerate}}
\end{lem}

\begin{proof}
This is exactly \cite[Lemma 1.1.8]{Kisin}. However, we give a proof for reader's convenience. 

Take $g\in G(\Qpbar)$ satisfying (i), and consider $[b_{\bs}(\mu')]\in B(T)$. Then Proposition \ref{ntkt} for the torus $T$ implies (ii). Let $[b']\in B(G)$ be the image of $[b_{\bs}(\mu')]$ under $B(T)\rightarrow B(G)$. Then $[b']$ lies in $B(G)_{\bs}$ by Proposition \ref{bseq}, and we have $\kappa_G([b'])=(\mu')^{\diamond}=\mu^{\diamond}$ by the functoriality of $\kappa_{G}$ for the natural inclusion $T\subset G$. Hence Proposition \ref{bskt} implies $[b']=[b]$. 
\end{proof}

\subsection{Local models}\label{lcmd}

Here we recall the theory of local models, which will control the singularities of integral models of Shimura varieties and Rapoport--Zink spaces. 

For a pair $(G,\{\mu\})$ of reductive connected group over $\Qp$ and a $G(\Qpbar)$-conjugacy class of a cocharacter of $G$, we call the field of definition of $\{\mu\}$ as the \emph{local reflex field} of $(G,\{\mu\})$. Note that it is a finite extension of $\Qp$. 

\begin{dfn}
A \emph{local model triple} is a triple $(G,\mu,\bx)$, where
\begin{itemize}
\item $G$ is a reductive connected group over $\Qp$,
\item $\mu \colon \G_{m,\Qpbar}\rightarrow G\otimes_{\Qp}\Qpbar$ is a minuscule cocharacter,
\item $\bx \in \cB(G,\Qp)$. 
\end{itemize}
\end{dfn}

In the sequel of this paper, we denote by $\mu_{2n}$ the minuscule cocharacter
\begin{equation*}
\G_m \rightarrow \GSp_{2n,\Qp};t\mapsto \diag(t^{-1,(n)},1^{(n)}),
\end{equation*}
where $a^{(n)}$ is $n$-th copies of $a$. Moreover, we write $\bx_{2n}\in \cB(\GSp_{2n},\Qp)$ for a unique hyperspecial vertex satisfying $\cG_{\bx_{2n}}=\GSp_{2n,\Zp}$. 

Under the above notations, we further assume the following: 
\begin{itemize}
\item[(I)] $G$ splits over a tamely ramified extension of $\Qp$ and $p\nmid \pi_1(G^{\der})$,
\item[(I\hspace{-0.5mm}I)] there is a closed immersion $i\colon G\hookrightarrow \GSp_{2m,\Qp}$ satisfying $i\circ \mu=\mu_{2n}$, $i(\bx)=\bx_{2n}$ and $i(G)$ contains all scalar matrices. 
\end{itemize}

Let $M^{\loc}(G,\mu,\bx)$ be the local model defined as \cite{Pappas2013}. It is a projective $O_{E}$-scheme equipped with an action of $\cG_{\bx,O_E}=\cG_{\bx}\otimes_{\Zp}O_E$. We have a description of $M^{\loc}(G,\mu,\bx)$ as follows. 

\textbf{Case 1. $(G,\mu,\bx)=(\GSp_{2n},\mu_{2n},\bx_{2n})$. }
In this case, the assumptions (I) and (I\hspace{-0.5mm}I) are satisfied. Then \cite[Proposition 7.3]{Pappas2013} implies that there is an isomorphism between $M^{\loc}(G,\mu,\bx)$ and the moduli space of maximally totally isotropic subspaces of $V_{2n,\Zp}^{*}$. 

In the sequel, we simply write $M^{\loc}(\GSp_{2n},\mu_{2n},\bx_{2n})$ for $M^{\loc}_{2n}$. 

\textbf{Case 2. General case. }
We denote by $P_{\mu^{-1}}$ the parabolic subgroup of $G$ associated to $\mu^{-1}$ in the sense of \cite[2.1.1]{Kisin2018}, and write $X_{\mu}$ for the $G$-homogeneous variety attached to $P_{\mu^{-1}}$. It is a smooth projective variety defined over $E$ that satisfies $X_{\mu}(\Qpbar)=G(\Qpbar)/P_{\mu^{-1}}(\Qpbar)$. Take a closed immersion $i\colon G\hookrightarrow \GSp_{2n,\Qp}$ as in (I\hspace{-0.5mm}I). Then it induces a closed immersion of $E$-scheme
\begin{equation*}
X_{\mu}\hookrightarrow M^{\loc}(\GSp_{2n},\mu_{2n},\bx_{2n})\otimes_{\Zp}E. 
\end{equation*}
Now we define $M_{G,\mu,\bx}$ as the scheme-theoretic closure of $X_{\mu}$ in $M^{\loc}(\GSp_{2n},\mu_{2n},\bx_{2n})\otimes_{\Zp}O_E$. Then \cite[4.1.5]{Kisin2018} asserts that there is a $\cG_{\bx,O_E}$-equivariant isomorphism $M^{\loc}(G,\mu,\bx)\cong M_{G,\mu,\bx}$. In particular, we obtain a closed immersion
\begin{equation*}
M^{\loc}(G,\mu,\bx)\hookrightarrow M^{\loc}_{2n}\otimes_{\Zp}O_{E}. 
\end{equation*}

\subsection{Deformations of $p$-divisible groups with tensors}\label{dfpd}

We keep the notations in Section \ref{lcmd}, Let $\nilp_{\Eb}$ be the category of schemes over $O_{\Eb}$ in which $p$ is locally nilpotent. For $S\in \nilp_{O_{\Eb}}$, we set $\overline{S}:=S\otimes_{\Zp}\Fp$. 

For a $p$-divisible group $X$ over $S\in \nilp_{O_{\Eb}}$, we write $\D(X)$ for the \emph{contravariant} Dieudonn{\'e} crystal of $X$. Then the relative Frobenius on $X\times_{S}\overline{S}$ induces the Frobenius map
\begin{equation*}
F\colon \sigma^{*}\D(X)\rightarrow \D(X)
\end{equation*}
In particular, if $S=\spec \Fpbar$, we define the (contravariant) Dieudonn{\'e} module of $X$ to be
\begin{equation*}
\D(X)(\Zpb):=\varprojlim_{n\in \Zpn}\D(X)(\Zpb/p^n). 
\end{equation*}
Note that it is a free $\Zpb$-module of finite rank equipped with Frobenius $F$ and Vershibung $V=pF^{-1}$. We write $\D(X)_{\Q}$ for the rational Dieudonn{\'e} module of $X$, that is, the $\Qpb$-vector space $\D(X)(\Zpb)\otimes_{\Z}\Q$ equipped with Frobenius $F$ (and Vershibung) induced by that on $\D(X)(\Zpb)$. Moreover, for $S=\spec \F_{p^r}$, we also define the Dieudonn{\'e} module $\D(X)(\Z_{p^r})$ as the same manner as the case for $S=\Fpbar$. 

On the other hand, let $\D(X)_{S}$ be the pull-back of $\D(X)$ to the Zariski site over $S$. Then $\D(X)_{S}$ is equipped with the Hodge filtration $\Fil^{1}(X)$ satisfying $\D(X)_{S}/\Fil^{1}(X)\cong \Lie(X^{\vee})$. We regard $\Lie(X^{\vee})$ as the $0$-th filtration on $\D(X)_{S}$. 

Now fix a closed immersion $i\colon G\hookrightarrow \GSp_{2n,\Qp}$ satisfying $i(\bx)=\bx_{2n}$. We regard $\cG_{\bx}$ as a subgroup of $\GSp_{2n,\Zp}$ by $i$. Moreover, take a finite collection $(s_{\alpha})$ in $V_{2n,\Zp}^{\otimes}$ whose point-wise stabilizer in $\GSp_{2n,\Zp}$ is $\cG_{\bx}$ (if $G=\GSp_{2n,\Qp}$, we take $(s_{\alpha})$ as the empty collection). On the other hand, as in Section \ref{lcmd}, we have a closed immersion
\begin{equation*}
M^{\loc}(G,\mu,\bx)\hookrightarrow M^{\loc}_{2n,O_E}:=M^{\loc}_{2n}\otimes_{\Zp}O_E. 
\end{equation*}
We regard $M^{\loc}(G,\mu,\bx)$ as the subscheme of $M^{\loc}_{2n,O_E}$ by the above closed immersion. 

Let $(X,\lambda)$ be a principally polarized $p$-divisible group over $\Fpbar$. We assume that there is a collection $(t_{\alpha,0})$ of $\D(X)(\Zpb)^{\otimes}$ and an isomorphism $\tau \colon \D(X)(\Zpb)\xrightarrow{\cong} V_{2n,\Zp}^{*}\otimes_{\Zp}\Zpb$ satisfying the following: 
\begin{itemize}
\item $\tau$ maps $t_{\alpha,0}$ to $s_{\alpha}\otimes 1$, and respects $\lambda_0$ and $\psi_{2n}$ up to $\Qpbt$-multiple, 
\item the point $y \in M^{\loc}_{2n}(\Fpbar)$ corresponding to the Hodge filtration $\Fil^0(X)$ of $\D(X)^{*}_{\Fpbar}$ belongs to $M^{\loc}(G,\mu,\bx)$. 
\end{itemize}

We define the deformation of the triple $(X,\lambda,(t_{\alpha}))$, which is a formal scheme over $\spf O_{\Eb}$ whose underlying space is a singuleton. 

\textbf{Case 1. $(G,\mu,\bx)=(\GSp_{2n},\mu_{2n},\bx_{2n})$. }
We consider the deformation $\Def(X,\lambda)$ of the polarized $p$-divisible group $(X,\lambda)$ in the usual sense. Then the Grothendieck--Messing theory implies that there is an isomorphism
\begin{equation*}
\Def(X,\lambda)\cong \widehat{M}^{\loc}_{2m,\Zpb,y},
\end{equation*}
which is defined as taking the pull-back of the Hodge filtration by $\tau$. 

\textbf{Case 2.~General case. }
We define $\Def(X,\lambda,(t_{\alpha,0}))$ as the following commutative diagram: 
\begin{equation*}
\xymatrix{
\Def(X,\lambda,(t_{\alpha,0}))\ar[r]^{\cong}\ar[d] & \widehat{M}^{\loc}(G,\mu,\bx)_{y} \ar[d]\\
\Def(X,\lambda)_{O_{\Eb}} \ar[r]^{\cong}& \widehat{M}^{\loc}_{2n,O_{\Eb},y},}
\end{equation*}
where the lower horizontal isomorphism is the base change to $O_{\Eb}$ of the isomorphism in Case 1. 


\section{Kisin--Pappas integral models}\label{kpim}

\subsection{Definition of Kisin--Pappas integral models}\label{dfkp}

We recall the definition of the integral models of Shimura varieties of Hodge type constructed as in \cite[\S 4.2]{Kisin2018}. 

First, we give a specific Shimura datum in the sense of \cite{Deligne1979}. Let $\bS:=\Res_{\C/\R}\G_{m,\C}$, the Weil restriction of $\G_{m,\C}$ to $\R$. For $m\in \Zpn$, let $S_{2n}^{\pm}$ be the $\GSp_{2n}(\R)$-conjugacy class of $\Hom_{\R}(\bS,\GSp_{2n,\R})$ containing
\begin{equation*}
\bS \rightarrow \GSp_{2n,\R};a+b\sqrt{-1}\mapsto 
\begin{pmatrix}
aE_{n}&-bE_{n}\\
bE_{n}&aE_{n}
\end{pmatrix}. 
\end{equation*}
Then $(\GSp_{2n,\Q},S_{2n}^{\pm})$ is a Shimura datum with reflex field $\Q$. 

Now let $(\bG,\bX)$ be a Shimura datum with reflex field $\E$. In this paper, we assume that $(\bG,\bX)$ is of Hodge type, that is, there is a closed immersion $\bi \colon (\bG,\bX) \hookrightarrow (\GSp_{2n,\Q},S_{2n}^{\pm})$. Moreover, we assume that $G:=\bG\otimes_{\Q}\Qp$ satisfies the condition (I) in Section \ref{lcmd}. Let $\bx \in \cB(\bG,\Qp)$, and define a compact open subgroup $K_p$ of $\bG(\Qp)$ as the \emph{full} stabilizer of $\bx$. Moreover, let $K^p$ be a neat compact open subgroup of $\bG(\A_{f}^{p})$, and put $K:=K_pK^p$. In the following, we give a definition of Kisin--Pappas integral model of the Shimura variety $\Sh_{K}(\bG,\bX)$ for a place $\mathfrak{p}$ of $\E$ lying above $p$. 

\textbf{Case 1. $(\bG,\bX)=(\GSp_{2m,\Q},S_{2n}^{\pm})$ and $K_p$ is hyperspecial. }
In this case, the condition (I) is affirmative. We define a Kisin--Pappas integral model $\sS_{K,2m}=\sS_{K}(\bG,\bX)$ as the moduli space of the functor which classifies prime-to-$p$ isogeny classes of triples $(A,\lambda,\overline{\eta}^p)$ for any connected noetherian scheme $S$ over $\Z_{(p)}$, where
\begin{itemize}
\item $A$ is an abelian scheme over $S$ of dimension $m$,
\item $\lambda$ is a prime-to-$p$ polarization of $A$,
\item $\overline{\eta}^p$ is a $K^p$-level structure on $A$. 
\end{itemize}
Note that $\sS_{K,2m}$ is formally smooth over $\Z_{(p)}$ with generic fiber $\Sh_{K}(\GSp_{2m,\Q},S_{2n}^{\pm})$. 

\textbf{Case 2. General case. }
Choose a closed immersion of Shimura datum $\bi \colon (\bG,\bX) \hookrightarrow (\GSp_{2m,\Q},S_{2n}^{\pm})$ satisfying $\bi(\bx)=\bx_{2n}$. Such a morphism exists by \cite[Proposition 1.3.3]{Kisin2018}. For a neat compact open subgroup ${K'}^p$ of $\GSp_{2n}(\A_{f}^{p})$, put $K':={K'}^p\GSp_{2n}(\Zp)$. Here we require that the morphism
\begin{equation*}
\Sh_{K}(\bG,\bX)\rightarrow \Sh_{K'}(\GSp_{2m,\Q},S_{2n}^{\pm})\otimes_{\Q}\E
\end{equation*}
induced by $\bi$ is a closed immersion. Note that such $K^p$ can be taken by \cite[Lemma (2.1.2)]{Kisin2010}. Let $\sS_{K',2n}$ be the $\Z_{(p)}$-scheme defined as in Case 1. 

For a place $\mathfrak{p}$ of $E$ above $p$, denote by $O_{\E,(\mathfrak{p})}$ the valuation subring of $\E$ with respect to $\mathfrak{p}$. We define a Kisin--Pappas integral model $\sS_{K}$ of $\Sh_{K}(\bG,\bX)$ to be the normalization of the scheme-theoretic closure of $\Sh_{K}(\bG,\bX)$ in $\sS_{K',2n}\otimes_{\Z_{(p)}}O_{\E,(\mathfrak{p})}$. By definition, $\sS_{K}$ is flat over $O_{\E,(\mathfrak{p})}$ with generic fiber $\Sh_{K}(\bG,\bX)$. 

\begin{rem}
It is non-trivial that whether $\sS_{K}(\bG,\bX)$ is independent of the choice of the embedding of Shimura datum $\bi \colon (\bG,\bX)\hookrightarrow (\GSp_{2n},S_{2n}^{\pm})$. When $K_p$ is hyperspecial, it was proved by \cite{Kisin2010} by giving a characterization by means of an ``extension property''. In general case, the independence of $\bi$ was proved by \cite{Zhang2019}. Moreover, if $K_p$ is a connected parahoric subgroup, that is, if $\cG_{x}=\cG_{x}^{\circ}$, Pappas gives a characterization on $\sS_{K}(\bG,\bX)$ by means of $(\cG_{x},M^{\loc}(G,\mu,\bx))$-displays. See \cite{Pappas2020} for more details. 
\end{rem}

\subsection{Newton stratification}\label{ntst}

Let $\sS_{K}$ be the Kisin--Pappas integral model associated to an embedding $\bi \colon (\bG,\bX)\hookrightarrow (\GSp_{2n,\Q},S^{\pm}_{2n})$ as in Section \ref{dfkp}. By construction, we obtain an abelian scheme $h\colon \cA_{K} \rightarrow \sS_{K}$ with prime-to-$p$ polarization $\lambda_{K}\colon \cA_{K}\rightarrow \cA_{K}^{\vee}$. More precisely, it is the pull-back of the universal abelian scheme over $\sS_{K,2n}$ along the canonical morphism $\sS_{K}\rightarrow \sS_{K,2n}$. We denote by $\cA_{K,z}$ the pull-back of $\cA_{K}$ by the morphism $z\colon S\rightarrow \sS_{K}$. 

Let $\bG_{\Z_{(p)}}$ be the closure of the image of the closed immersion $\bG \xrightarrow{\bi} \GSp_{2n,\Q}\hookrightarrow \GSp_{2n,\Z_{(p)}}$. Then $\bG_{\Z_{(p)}}$ is smooth over $\Z_{(p)}$, and that there is an isomorphism $\bG_{\Z_{(p)}}\otimes_{\Z_{(p)}}\Zp \cong \cG_{\bx}$. Moreover, there is a finite collection $(s_{\alpha})$ of $V_{2n,\Z_{(p)}}^{\otimes}$ whose point-wise stabilizer in $\GSp(V_{2n,\Z_{(p)}})$ equals $\bG_{\Z_{(p)}}$. 

Let $z\in \sS_{K}(\F_p^r)$ where $r\in \Zpn$. As in \cite[\S 9.2]{Zhou2020}, there are a finite collection $(t_{\alpha,0,z})$ in $\D(\cA_{K,z})(\Z_{p^r})^{\otimes}$ and an isomorphism
\begin{equation*}
\tau_z\colon \D(\cA_{K,z})(\Z_{p^r}) \xrightarrow{\cong} V_{2n,\Zp}^{*}\otimes_{\Zp}\Z_{p^r}
\end{equation*}
that satisfy the conditions (i) and (ii) in Section \ref{dfpd}. Hence we can consider the deformation $\Def(\cA_{K,z}[p^{\infty}],\lambda_{z},(t_{\alpha,0,z}))$ of $(\cA_{z}[p^{\infty}],\lambda_{z},(t_{\alpha,0,z}))$. It gives a formal neighborhood of $z$ as follows. 

\begin{prop}\label{dfp1} (\cite[Proposition 4.2.2, Corollary 4.2.4]{Kisin2018})
\emph{Let $z\in \sS_{K}(\Fpbar)$. Then there is an isomorphism $\widehat{\sS}_{K,z}\cong \Def(\cA_{K,z}[p^{\infty}],\lambda_{K,z},(t_{\alpha,0,z}))$. }
\end{prop}

Next, we are going to define Newton strata of $\sSbar_{K}$. 
\begin{dfn}\label{ptsc}
Under the notations above, we define $b_z\in G(\Q_{p^r})\subset G(\Qpb)$ as the element such that $b_z\sigma$ corresponds to the Frobenius on $\D(\cA_{K,z})_{\Q}$ under $\tau_z$. 
\end{dfn}
By construction, $[b_z]\in B(G)$ is uniquely determined. Moreover, it lies in $B(G,\{\mu\})$ by \cite[\S 8]{Zhou2020}. 

\begin{dfn}
For $[b]\in B(G,\{\mu\})$, put
\begin{equation*}
\sSbar_{K,[b]}:=\{y\in \sSbar_{K}\mid [b_y]=[b]\}. 
\end{equation*}
We call it as the \emph{Newton stratum} of $\sSbar_{K}$ with respect to $[b]$. 
\end{dfn}

The stratum $\sSbar_{K,[b]}$ is locally closed in $\sS_{K}$ by \cite[Corollary 4.12 (1)]{Hamacher2019}. Moreover, if $[b]$ is basic, it is non-empty and closed in $\sS_{K}$ by \cite[Theorem 1.3.13]{Kisin}. 

\subsection{Local model diagrams for Kisin--Pappas integral models}

We keep the notations as in Section \ref{ntst}. For $h\in \bX$, let $\mu_h$ be the composite
\begin{equation*}
\G_{m,\C}\rightarrow \bS \otimes_{\R}\C \xrightarrow{h} \bG\otimes_{\Q}\C,
\end{equation*}
where $\G_{m,\C}\rightarrow \bS \otimes_{\R}\C$ is induced by a natural inclusion $\G_{m,\R}\hookrightarrow \bS$. Then we obtain a cocharacter $\mu$ of $G$ by $\mu_h^{-1}$ and the fixed isomorphism $\Qpbar \cong \C$. Note that the conjugacy class $\{\mu\}$ of $\mu$ is independent of $\mu$, and the closed immersion $i\colon G\hookrightarrow \GSp_{2n,\Qp}$ induced by an embedding $\bi$ satisfies (I\hspace{-0.5mm}I) for $\mu$. Hence we can consider the local model $M^{\loc}(G,\mu,\bx)$ for $(G,\mu,\bx)$. Let $O_E$ be the integer ring of the local reflex field of $(G,\{\mu\})$. Then $O_{E}$ is the completion of $O_{\E,(\mathfrak{p})}$. 

Now we recall the local model diagram constructed by Kisin and Pappas. Set
\begin{equation*}
\cV_{\cA_{K}/\sS_{K}}:=\rR^{1}h_{*}\Omega_{\cA_{K}/\sS_{K}}^{\bullet}. 
\end{equation*}
It is equipped with the Hodge filtration $\Fil^{1}(\cA_{K,z})\subset \cV_{\cA_{K}/\sS_{K}}$ which induces an isomorphism $\cV_{\cA_{K}/\sS_{K}}/\Fil^{1}(\cA_{K,z})\cong \Lie(\cA_{K,z}^{\vee})$. We denote by $(s_{\alpha,\dR})$ be a finite collection in $\cV_{\cA_{K}/\sS_{K}}^{\otimes}$ as constructed in \cite[Proposition 4.2.6]{Kisin2018}. 

\begin{dfn}\label{sktl}
We define an $O_E$-scheme $\widetilde{\sS}_{K,O_E}$ as the moduli space of the functor that parametrizes pairs $(z,f)$ for any $O_{E}$-scheme $S$, where
\begin{itemize}
\item $z\in \sS_{K}(S)$,
\item $f\colon \cV_{\cA_{K}/\sS_{K},z} \xrightarrow{\cong}V_{\Z_{(p)}}^{*}\otimes_{\Z_{(p)}} \O_S$ is an isomorphism which maps $s_{\alpha,\dR,z}$ to $s_{\alpha}\otimes 1$. 
\end{itemize}
\end{dfn}

Furthermore, we define two morphisms from $\widetilde{\sS}_{K,O_E}$ as follow: 
\begin{equation*}
\pi \colon \widetilde{\sS}_{K,O_E}\rightarrow \sS_{K,O_E};(z,f)\mapsto z,\quad
q\colon \widetilde{\sS}_{K,O_E}\rightarrow M^{\loc}_{2n,O_E};(z,f)\mapsto f(\Fil^1(\cA_{K,z})). 
\end{equation*}

\begin{prop}\label{lmds}(\cite[Theorem 4.2.7]{Kisin2018})
\emph{
\begin{enumerate}
\item The morphism $q\colon \widetilde{\sS}_{K,O_E}\rightarrow M^{\loc}_{2n,O_E}$ factors through $M^{\loc}(G,\mu,\bx)\hookrightarrow M^{\loc}_{2m,O_E}$. 
\item The diagram
\begin{equation*}
\xymatrix{
&\widetilde{\sS}_{K,O_E}\ar[ld]_{\pi}\ar[rd]^{q}&\\
\sS_{K,O_{E}}&&M^{\loc}(G,\mu,\bx),}
\end{equation*}
satisfies the following: 
\begin{itemize}
\item $\pi$ is a $\cG_{\bx,O_E}$-torsor,
\item $q$ is $\cG_{\bx,O_E}$-equivariant and smooth of relative dimension $\dim(G)$. 
\end{itemize}
\end{enumerate}}
\end{prop}

\section{Rapoport--Zink spaces}\label{rzsp}

In this section, we recall the definitions of the Rapoport--Zink spaces of Hodge type with Bruhat--Tits level structure constructed by \cite[\S \S 5.3--5.4]{Hamacher2019}. 

First, we give the notion of the embedded Rapoport--Zink datum; cf.~\cite[\S 5.1]{Hamacher2019}. 
\begin{dfn}
An \emph{embedded Rapoport--Zink datum} is a quintuple $(G,\mu,\bx,b,i)$, where
\begin{enumerate}\label{ebrz}
\item $G$ is a reductive connected group over $\Qp$,
\item $\mu \colon \G_{m,\Qpbar}\rightarrow G\otimes_{\Qp}\Qpbar$ is a minuscule cocharacter,
\item $\bx \in \cB(G,\Qp)$,
\item $b\in \bigcup_{w\in \Adm(\mu)}\cG_{\bx}^{\circ}(\Zpb)\sigma(w)\cG_{\bx}^{\circ}(\Zpb)$ is a decent element,
\item $i\colon G\hookrightarrow \GSp_{2n,\Qp}$ is a closed immersion,
\end{enumerate}
satisfying the following conditions: 
\begin{itemize}
\item $i\circ \mu=\mu_{2n}$, $[i(b)]\in B(\GSp_{2n,\Qp},\mu_{2n})$ and $i(\bx)=\bx_{2n}$. 
\end{itemize}
\end{dfn}

\begin{rem}
By \cite[Theorem 2.1]{He2016}, $b\in G(\Qpb)$ as in Definition \ref{ebrz} satisfies $[b] \in B(G,\{\mu\})$. Hence $(G,\{\mu\},[b])$ is a local Shimura datum in the sense of \cite[Definition 5.1]{Rapoport2014b}. 
\end{rem}

Set $\Kb_p:=\cG_{\bx}(\Zpb)$, which is a compact open subgroup of $G(\Qpb)$, and put
\begin{equation*}
X_{\mu}(b)_{\Kb_p}(\Fpbar):=\left\{g\Kb_p\in G(\Qp)/\Kb_p\,\middle| \, g^{-1}b\sigma(g)\in \bigcup_{w\in \Adm(\mu)}\Kb_p w\Kb_p \right\}. 
\end{equation*}
As pointed out in \cite[5.2]{Hamacher2019}, we can regard it as a perfect subscheme of the Witt vector affine Grassmanian of $\cG_{\bx}$ by using \cite[Corollary 9.6]{Bhatt2017}. 

We give a description of $X_{\mu}(b)_{\Kb_p}(\Fpbar)$ by means of $p$-divisible groups with additional structure. Let $(\X_b,\lambda_b)$ be a polarized $p$-divisible group admitting an isomorphism 
\begin{equation*}
\tau_b\colon \D(\X_b)(\Zpb)\xrightarrow{\cong} V_{2n,\Z_{(p)}}^{*}\otimes_{\Z_{(p)}}\Zpb
\end{equation*}
such that the Frobenius on $\D(\X_b)(\Zpb)$ corresponds to $b\sigma^{*}$ under $\tau_{b}$. Then put $t_{b,\alpha}:=\tau_b^{-1}(s_{\alpha}\otimes 1)$. 

\begin{prop}\label{adrz}
\emph{There is an isomorphism between $X_{\sigma(\mu)}(b)_{\Kb_p}(\Fpbar)$ and the set of isomorphism classes of triples $(X,\lambda,\rho)$, where
\begin{itemize}
\item $X$ is a $p$-divisible group over $\Fpbar$,
\item $\lambda \colon X\rightarrow X^{\vee}$ is a principal polarization, 
\item $\rho \colon \X_{b} \rightarrow X$ is a quasi-isogeny,
\end{itemize}
that satisfy the following: 
\begin{itemize}
\item $\rho^{\vee}\circ \lambda \circ \rho=c\lambda_{0}$ for some $c\in \Qpt$,
\item $\rho_{*}(t_{b,\alpha})\in \D(X)(\Zpb)^{\otimes}[p^{-1}]$ lies in $\D(X)(\Zpb)^{\otimes}$ for any $\alpha$. 
\end{itemize}}
\end{prop}

\begin{proof}
This is a consequence of \cite[Definition/Lemma 5.5]{Hamacher2019}. 
\end{proof}

From now on, we assume (I) as in Section \ref{lcmd} and the following: 
\begin{itemize}
\item[(I\hspace{-0.5mm}I\hspace{-0.5mm}I)] there is an embedding of Shimura datum $\bi\colon (\bG,\bX)\hookrightarrow (\GSp_{2n,\Q},S_{2n}^{\pm})$ which induces $i$. 
\end{itemize}
Moreover, we require the following property: 

\begin{axm}\label{lift}(\cite[Axiom A]{Hamacher2019})
Under the notations as above, let $(z,j)$ be a pair where $z\in \sSbar_{K,[b]}(\Fpbar)$ and $j\colon \X_b\rightarrow \cA_{K,z}[p^{\infty}]$ is a quasi-isogeny which commutes with the polarizations up to $\Qpt$-multiple. We define a map
\begin{equation*}
\theta'_{K,(z,j)}\colon X_{\sigma(\mu)}(b)_{\Kb_p}(\Fpbar)\rightarrow \sS_{K',2n}(\Fpbar),
\end{equation*}
as sending $(X,\lambda,\rho)$ to the triple consisting of the abelian variety $\cA_{z}/\Ker(\rho \circ j^{-1})$ and the polarization and the $K^p$-level structure induced by $\lambda_z$ and $\overline{\eta}_{z}^{p}$ respectively. Then $\theta'_{(z,j)}$ uniquely factors through $\sS_{K}(\Fpbar)\rightarrow \sS_{K',2n}(\Fpbar)$. 
\end{axm}

Axiom \ref{lift} is not yet known is general. However, there are many affirmative results on this axiom as follows. 

\begin{prop}\label{lfok}
\emph{Axiom \ref{lift} holds if one of the following hold: 
\begin{enumerate}
\item $K_p$ is hyperspecial (in particular $G$ is unramified), 
\item $[b]$ is basic,
\item $G$ is residually split, that is, $G$ and $G\otimes_{\Qp}\Qpb$ have the same split ranks, 
\item $G$ is quasi-split and $K_p$ is absolutely special connected parahoric. 
\end{enumerate}}
\end{prop}

\begin{proof}
(i): This is proved by \cite[Proposition (1.4.4)]{Kisin2017}. 

(ii), (iii): If $\cG_{\bx}=\cG_{\bx}^{\circ}$, then these are contained in \cite[Proposition 7.8]{Zhou2020}. The proof in the general case is also the same as the previous case by \cite[Remark 5.4]{Hamacher2019}. 

(iv): This is \cite[Theorem A.4.3]{vanHoften2020}. 
\end{proof}

In the following, we recall the construction of the Rapoport--Zink spaces for $(G,\mu,\bx,b,i)$ in \cite[\S 5.3]{Hamacher2019} when the assumptions (I), (I\hspace{-0.5mm}I\hspace{-0.5mm}I) and Axiom \ref{lift} are satisfied. It will be a formal scheme, which is locally formally of finite type over $\spf O_{\Eb}$. 

\textbf{Case 1.~$(G,\mu,\bx,b,i)=(\GSp_{2n},\mu_{2n},\bx_{2n},b,\id)$. }
In this case, we have $\Eb=\Qpb$. Moreover, it is clear that (I), (I\hspace{-0.5mm}I\hspace{-0.5mm}I) and Axiom \ref{lift} are affirmative. We define the Rapoport--Zink space 
\begin{equation*}
\RZ_{2n,b}=\RZ_{K_p}(\GSp_{2n},\mu_{2n},b)
\end{equation*}
as the moduli space of the functor which sends $S\in \nilp_{\Zpb}$ to the isomorphism classes of triples $(X,\lambda,\rho)$, where
\begin{itemize}
\item $X$ is a $p$-divisible group of dimension $m$ and height $2m$ over $S$,
\item $\lambda \colon X\rightarrow X^{\vee}$ is a principal polarization,
\item $\rho \colon \X_{b}\otimes_{\Fpbar}\Sbar \rightarrow X\times_{S}\Sbar$ is a quasi-isogeny satisfying $\rho^{\vee}\circ \lambda \circ \rho=c(\rho)\lambda_{b}$ for some locally constant function $c(\rho)\colon \Sbar \rightarrow \Qpt$. 
\end{itemize}
Note that the representability of $\RZ_{2n,b}$ is a result of \cite[Theorem 3.25]{Rapoport1996b}. 

Let $(z,j)$ be a pair as in Axiom \ref{lift}. Then we have a morphism of $\Zpb$-formal schemes
\begin{equation*}
\Theta_{K,(z,j)} \colon \RZ_{2n,b}\rightarrow \sS_{K,2n}
\end{equation*}
defined as the same manner as $\theta_{K,(z,j)}$. See also \cite[\S 5.3]{Hamacher2019}. It is functorial with respect to $K^p$. We call this map as the \emph{$p$-adic uniformization map}. 

\textbf{Case 2.~General case. }
Consider the fiber product
\begin{equation*}
\xymatrix@C=46pt{
\RZ_{K_p}(G,\mu,b)^{\diamond}\ar[d] \ar[r]^{\quad \quad \Theta_{K,(z,j)}^{\diamond}} & \sS_{K} \ar[d]\\
\RZ_{2n,i(b)}\ar[r]^{\Theta_{K',(i_{*}(z),j)}} &\sS_{K',2n}. }
\end{equation*}
Let $(X^{\diamond},\lambda^{\diamond},\rho^{\diamond})$ be the pull-back of the universal object over $\RZ_{2n,i(b)}$ by $\Theta_{K,\widetilde{z}}^{\diamond}$. Furthermore, for $y\in \RZ_{K_p}(G,\mu,\bx)^{\diamond}(\Fpbar)$, we write $t_{\alpha,y}^{\diamond}$ for the element corresponding to $t_{b,\alpha}$ under the isomorphism $\D(X_{y}^{\diamond})(\Zpb)^{\otimes}[p^{-1}] \xrightarrow{\cong} \D(\X_{b})^{\otimes}[p^{-1}]$ induced by $\rho^{\diamond}$. Then \cite[\S 5.3]{Hamacher2019} implies that there is a unique open and closed formal subscheme of $\RZ_{K_p}(G,\mu,b)^{\diamond}$ satisfying the following two conditions:
\begin{enumerate}
\item there is an isomorphism of perfect $\Fpbar$-schemes $\RZ_{K_p}(G,\mu,b)^{\red,p^{-\infty}}\cong X_{\sigma(\mu)}(b)_{\Kb_p}$, where the left-hand side is the perfection of the underlying reduced subscheme of $\RZ_{K_p}(G,\mu,b)$, 
\item for any $y\in X_{\sigma(\mu)}(b)_{\Kb_p}(\Fpbar)$, there is an isomorphism $\widehat{\RZ}_{K_p}(G,\mu,b)_{y}\cong \Def(X^{\diamond}_y,\lambda^{\diamond}_y,t^{\diamond}_{\alpha,y})$. 
\end{enumerate}
Moreover, the composite
\begin{equation*}
\RZ_{K_p}(G,\mu,b)\hookrightarrow \RZ_{K_p}(G,\mu,b)^{\diamond}\rightarrow \RZ_{2n,i(b)}
\end{equation*}
becomes a closed immersion. This implies that $\RZ_{K_p}(G,\mu,b)$ depends only on the embedded Rapoport--Zink datum $(G,\mu,\bx,b,i)$. We call the formal scheme $\RZ_{K_p}(G,\mu,b)$ as the \emph{Rapoport--Zink space} for $(G,\mu,\bx,b,i)$. 

By construction, we obtain the \emph{uniformization map}
\begin{equation*}
\Theta_{K,(z,j)} \colon \RZ_{K_p}(G,\mu,b)\rightarrow \sS_{K}
\end{equation*}
by restricting $\Theta_{K,(z,j)}^{\diamond}$ to the closed formal subscheme $\RZ_{K_p}(G,\mu,b)\hookrightarrow \RZ_{K_p}(G,\mu,b)^{\diamond}$. 

\begin{rem}
The sign convention on $\{\mu\}$ is different from \cite{Hamacher2019}. However, it is compatible with those of \cite{Kisin2010}, \cite{Pappas2013} and \cite{Zhou2020}. 
\end{rem}

In the end of this section, we recall the left action of $J_{b}(\Qp)$ on $\RZ_{K_p}(G,\mu,b)$ as introduced in \cite[\S 5.4]{Hamacher2019}, where $J_{b}$ be an algebraic group over $\Qp$ as defined in Section \ref{gisc}. First, note that $i$ induces a closed immersion
\begin{equation*}
J_{b}\hookrightarrow J_{i(b)}. 
\end{equation*}
On the other hand, the isomorphism $\tau_{b}\colon \D(\X_b)(\Zpb)\xrightarrow{\cong}V_{2n,\Zp}^{*}\otimes_{\Zp}\Zpb$ as above induces an anti-homomorphism
\begin{equation*}
J_{b}(\Qp)\rightarrow \End(\X_b);g\mapsto g_{\X_b}. 
\end{equation*}

\textbf{Case 1.~$(G,\mu,\bx,b,i)=(\GSp_{2n},\mu_{2n},\bx_{2n},b,\id)$. }
The action on $J_{b}(\Qp)$ on $\RZ_{2n,b}$ is the same as \cite[Definition 3.22]{Rapoport1996b}. More precisely, it is defined as
\begin{equation*}
a_{b}\colon J(\Qp)\times \RZ_{2n,b}\rightarrow \RZ_{2n,b};(g,(X,\lambda,\rho))\mapsto (X,\lambda,\rho \circ g_{\X_b}^{-1}). 
\end{equation*}

\textbf{Case 2.~General case. }
By the proof of \cite[Lemma 5.12]{Hamacher2019}, we can written the action of $J_{b}(\Qp)$ on $\RZ_{K_p}(G,\mu,b)$ as the top horizontal map of the following commutative diagram: 
\begin{equation*}
\xymatrix{
J_{b}(\Qp)\times \RZ_{K_p}(G,\mu,b)\ar[r] \ar[d]& \RZ_{K_p}(G,\mu,b) \ar[d]\\
J_{i(b)}(\Qp)\times \RZ_{2n,i(b)}\ar[r]^{\quad \quad a(i(b))}& \RZ_{2n,i(b)}. }
\end{equation*}

\section{Local model diagrams for Rapoport--Zink spaces}\label{lmdr}

Let $(G,\mu,\bx,b,i)$ be an embedded Rapoport--Zink space satisfying (I), (I\hspace{-0.5mm}I\hspace{-0.5mm}I) and Axiom \ref{lift}. We study the singularity of the Rapoport--Zink space $\RZ_{K_p}(G,\mu,b)$ for $(G,\mu,\bx,b,i)$. First, note that the condition (I\hspace{-0.5mm}I\hspace{-0.5mm}I) implies (I\hspace{-0.5mm}I), and hence we can consider the local model $M^{\loc}(G,\mu,\bx)$ for $(G,\mu,\bx)$. We denote by $\widehat{M}^{\loc}(G,\mu,\bx)_{O_{\Eb}}$ the $p$-adic completion of $M^{\loc}(G,\mu,\bx)\otimes_{O_E}O_{\Eb}$. 

On the other hand, let $\widehat{\sS}_{K,O_{\Eb}}$ be the $p$-adic completion of $\sS_{K,O_{\Eb}}$, and $\sX_{K}$ the $p$-divisible group over $\widehat{\sS}_{K,O_{\Eb}}$ associated to $\cA_{K}$. Then we obtain a sheaf $W(\O_{\widehat{\sS}_{K,O_{\Eb}}})$ over $\widehat{\sS}_{K,O_{\Eb}}$ satisfying
\begin{equation*}
\Gamma(\spf \mathfrak{S},W(\O_{\widehat{\sS}_{K,O_{\Eb}}}))=W(\mathfrak{S})
\end{equation*}
for any open formal subscheme $\spf \mathfrak{S}$ of $\widehat{\sS}_{K,O_{\Eb}}$. Moreover, there is a locally free $W(\O_{\widehat{\sS}_{K,O_{\Eb}}})$-module $P(\sX_{K})$ satisfying
\begin{equation*}
\Gamma(\spf \mathfrak{S},P(\sX_{K}))=\D(\sX_{K})_{W(\mathfrak{S})}^{*}
\end{equation*}
for any open formal subscheme $\spf \mathfrak{S}$ of $\widehat{\sS}_{K,O_{\Eb}}$. Under this notations, Hamacher and Kim constructed a finite family $(t_{\alpha})$ in $P(\sX_{K})^{\otimes}$ (\cite[Proposition 4.6]{Hamacher2019}). For $S\in \nilp_{O_{\Eb}}$ and $y\in \RZ_{K_p}(S)$, we denote by $t_{\alpha}\!\!\mid_{S}\,\in \D(X_y)_{S}^{\otimes}$ the pull-back of $t_{\alpha}$. 

\begin{dfn}
We define a formal scheme $\widetilde{\RZ}_{K_p}$ over $\spf O_{\Eb}$ as the moduli space of the functor that classifies pairs $(y,f)$ for $S\in \nilp_{O_{\Eb}}$, where
\begin{itemize}
\item $y\in \RZ_{K_p}(G,\mu,b)(S)$,
\item $f\colon \D(X_y)_{S} \xrightarrow{\cong}V_{2n,\Zp}^{*}\otimes_{\Zp}\O_S$ is an isomorphism which maps $t_{\alpha}\!\!\mid_S$ to $s_{\alpha}\otimes 1$. 
\end{itemize}
\end{dfn}

We define two morphisms from $\widetilde{\RZ}_{K_p}(G,\mu,b)$ as follow: 
\begin{align*}
\pi_b &\colon \widetilde{\RZ}_{K_p}(G,\mu,b)\rightarrow \RZ_{K_p}(G,\mu,b);(y,f)\mapsto y,\\
q_b &\colon \widetilde{\RZ}_{K_p}(G,\mu,b)\rightarrow M^{\loc}_{2m,O_{\Eb}};(y,f)\mapsto f(\Fil^1(X_y)). 
\end{align*}

\begin{thm}\label{llmd}
\emph{
\begin{enumerate}
\item The morphism $q_b\colon \widetilde{\RZ}_{K_p}(G,\mu,b)\rightarrow M^{\loc}_{2m,O_{\Eb}}$ factors through $M^{\loc}(G,\mu,\bx)\hookrightarrow M^{\loc}_{2m,O_{\Eb}}$. 
\item The diagram
\begin{equation*}
\xymatrix{
&\widetilde{\RZ}_{K_p}(G,\mu,b)\ar[ld]_{\pi_b} \ar[rd]^{q_b} & \\
\RZ_{K_p}(G,\mu,b)& &\widehat{M}^{\loc}(G,\mu,x),}
\end{equation*}
satisfies the following: 
\begin{itemize}
\item $\pi_b$ is a $\cG_{\bx,O_E}$-torsor, 
\item $q_b$ is $\cG_{\bx,O_E}$-equivariant and formally smooth of relative dimension $\dim(G)$. 
\end{itemize}
\end{enumerate}}
\end{thm}

\begin{proof}
Let $S\in \nilp_{O_{\Eb}}$ and $y\in \RZ_{K_p}(S)$. By the construction of $(t_{\alpha})$ as in \cite[\S 4.3]{Hamacher2019}, there is an isomorphism 
\begin{equation*}
c_y\colon \D(X_y)_{S}\xrightarrow{\cong} \cV_{K}
\end{equation*}
which maps $t_{\alpha}\!\!\mid_{S}$ to $t_{\alpha,\dR}$ for any $\alpha$. Note that the above isomorphism is functorial with respect to $y$ and $S$. Hence the diagram
\begin{equation*}
\xymatrix{
\widetilde{\RZ}_{K_p}(G,\mu,b) \ar[r]^{\quad \quad \widetilde{\Theta}_{K,(z,j)}} \ar[d]_{q_b} & \widetilde{\sS}_{K} \ar[d]_{q} \\
\RZ_{K_p}(G,\mu,b)\ar[r]^{\quad \quad \Theta_{K,(z,j)}} &\sS_{K}, }
\end{equation*}
where $\widetilde{\Theta}_{K,(z,j)}(y,f):=(\Theta_{K,(z,j)},f\circ c_y)$, is Cartesian. 

(i): This follows from Proposition \ref{lmds} (i). 

(ii): The morphism $q_b$ is a $G$-torsor, since $q\colon \widetilde{\sS}_{K}\rightarrow \sS_{K}$ is so by Proposition \ref{lmds} (ii). On the other hand, $\widetilde{\Theta}_{K,(z,j)}$ is formally smooth, since $\Theta_{K,(z,j)}$ induces an isomorphism $\widehat{\RZ}_{K_p}(G,\mu,b)_{y}\cong \widehat{\sS}_{K,\Theta_{K,(z,j)}(y)}$ for any $y\in \RZ_{K_p}(\Fpbar)$ by construction of $\RZ_{K_p}$. Hence so is the composite
\begin{equation*}
\widetilde{\RZ}_{K_p}(G,\mu,b)\xrightarrow{\widetilde{\Theta}_{K,(z,j)}} \widetilde{\sS}_{K}\xrightarrow{q} M^{\loc}(G,\mu,\bx). 
\end{equation*}
This morphism induces a that of formal schemes $q_b$ as desired.  
\end{proof}

\section{$p$-adic uniformization theorem}\label{padc}

In this section, we consider basic loci of Shimura varieties. We keep the notations in Section \ref{kpim}, and let $b_0\in \bigcup_{w\in \Adm(\mu)}\cG_{\bx}^{\circ}(\Zpb)\sigma(w)\cG_{\bx}^{\circ}(\Zpb)$ be a decent element in the sense of \cite[Definition 1.6]{Rapoport1996b} such that $[b_0]\in B(G,\{\mu\})$ is the unique basic element. This is possible by \cite[Lemma 5.2]{Hamacher2019}. Then Axiom \ref{lift} holds for $z$ and $b_0$ by Proposition \ref{lfok} (ii). Hence we can consider the Rapoport--Zink space $\RZ_{K_p}(G,\mu,b_0)$ for $(G,\mu,\bx,b_0,i)$.  

We denote by $(\sShat_{K,O_{\Eb}})_{/\sSbar_{K,[b_0]}}$ the completion of $\sS_{K,O_{\Eb}}$ along $\sSbar_{K,[b_0]}$. 

\begin{thm}\label{paut}
\emph{For $z\in \sSbar_{K,[b_0]}$, there is an isomorphism of $O_{\Eb}$-formal schemes
\begin{equation*}
\bI_{z}(\Q)\backslash \RZ_{K_p}(G,\mu,b_0)\times \bG(\A_{f}^{p})/K^p\cong (\sShat_{K,O_{\Eb}})_{/\sSbar_{K,[b]}}. 
\end{equation*}
Here $\bI_{z}$ is the inner form of $\bG$ which is anisotropic modulo center satisfying
\begin{equation*}
\bI_{z}\otimes_{\Q}\Ql \cong
\begin{cases}
\bG \otimes_{\Q}\Ql &\text{if }\ell \neq p,\\
J_{b_0}&\text{if }\ell=p. 
\end{cases}
\end{equation*}}
\end{thm}

\begin{rem}
Let $\{g_1,\ldots,g_m\}$ be a complete representative of $\bI_{z}(\Q)\backslash \bG(\A_f^p)/K^p$, and set $\Gamma_i:=\bI_{z}(\Q)\cap (J_{b_0}(\Qp)\times g_iK^pg_i^{-1})$, which is regarded as a subgroup of $\bI_{z_0}(\Qp)\cong J_{b_0}(\Qp)$. Then, as in \cite[Theorem 6.30]{Rapoport1996b}, $\bI_{z}(\Q)\backslash \RZ_{K_p}(G,\mu,b_0)\times \bG(\A_{f}^{p})/K^p$ is of the form $\coprod_{i=1}^{m}\Gamma_{i}\backslash \RZ_{K_p}(G,\mu,b_0)$, and $\Gamma_i$ is discrete, torsion-free and cocompact modulo center. 
\end{rem}

We prove Theorem \ref{paut} in the sequel. 

\begin{lem}\label{pwct}
\emph{Let $z_1,z_2\in \sS_{K,[b_0]}(\F_{p^r})$, and $b_{z_1},b_{z_2}\in G(\Q_{p^r})$ be as Definition \ref{ptsc}. We denote by $Z^{\circ}$ the identity component of the center of $G$. Then there is $m\in \Zpn$ such that
\begin{equation*}
b_{z_1}\cdots \sigma^{rm-1}(b_{z_1})=b_{z_1}\cdots \sigma^{rm-1}(b_{z_1})\in Z^{\circ}(\Qp). 
\end{equation*}}
\end{lem}

\begin{proof}
Since $[b_0]$ is basic, there are $s\in \Zpn$ and $m'\in s\Z$ such that $rs\nu_{G}([b_0])$ is a cocharacter of $Z$ defined over $\Qp$ and
\begin{equation*}
b_{z}\cdots \sigma^{rm'-1}(b_{z})=(rs\nu_{G}([b_0]))(p)=b_{z'}\cdots \sigma^{rm'-1}(b_{z'})\in Z(\Qp). 
\end{equation*}
Moreover, there is $t\in \Zpn$ satisfying $(rst\nu_{G}([b_0]))(p)\in Z^{\circ}(\Qp)$. Hence $m:=m't$ satisfies the desired assertion. 
\end{proof}

Let $\ell \neq p$ be a prime. As in \cite[(3.4.2)]{Kisin2010}, there is a finite family of tensors $(t_{\alpha,\ell,z})$ in $(T_{\ell}\cA_{K,z})^{\otimes}\cong \rH^1_{\et}(\cA_{K,z},\Ql)^{\otimes}$ such that the isomorphism 
\begin{equation*}
\eta_{\ell,z}\colon V_{2n,\Z_{(p)}}\otimes_{\Z_{(p)}}\Ql \xrightarrow{\cong} T_{\ell}\cA_{K,z}\otimes_{\Z}\Q
\end{equation*}
induced by the $K^p$-level structure attached to $z$ maps $t_{\alpha,\ell,z}$ to $s_{\alpha}\otimes 1$ for any $\alpha$. We define an algebraic group $I_{\ell,z}$ over $\Ql$ as the subgroup of $\GL(\rH_{\et}^{1}(\cA_{K,z},\Ql))$ which stabilizes $t_{\alpha,\ell,z}$ for all $\alpha$ and $\gamma_{\ell}^{m}$ for sufficiently large $m\in \Zpn$. Then the isomorphism $\eta_{\ell,z}$ induces an injection 
\begin{equation*}
c_{\ell,z}\colon I_{\ell,z}\hookrightarrow \bG\otimes_{\Q}\Ql. 
\end{equation*}

On the other hand, let $\Aut_{\Q}(\cA_{K,z})$ be the algebraic group over $\Q$ to be
\begin{equation*}
\Aut_{\Q}(\cA_{K,z})(R)=(\End(\cA_{K,z})\otimes_{\Z}R)^{\times}
\end{equation*}
for any $\Q$-algebra $R$. Moreover, we define an algebraic group $\bI_{z}$ over $\Q$ as the subgroup of $\Aut_{\Q}(\cA_{K,z})$ which fixes all $s_{\alpha,\ell,z}$ and $s_{\alpha,0,z}$. Then there are injections
\begin{equation*}
\iota_{\ell,z}\colon \bI_{z}\otimes_{\Q}\Ql \hookrightarrow I_{\ell,z},\quad \iota_{p,z}\colon \bI_{z}\otimes_{\Q}\Qp \hookrightarrow J_{b}. 
\end{equation*}

\begin{lem}\label{exmt}
\emph{
\begin{enumerate}
\item The $\R$-group $\bI_{z}\otimes_{\Q}\R$ is anisotropic modulo center. 
\item For a prime $\ell \neq p$, $\bI_{z}\otimes_{\Q}\Ql$ contains the connected component of $I_{\ell,z}$. In particular, the ranks of $\bG$ and $\bI_{z}$ are equal. 
\end{enumerate}}
\end{lem}

\begin{proof}
The proof is the same as \cite[Corollary (2.1.7)]{Kisin2017}. 
\end{proof}

For $z\in \sS_{K}(\Fpbar)$, let $\theta_{K,(z,j)}\colon X_{\sigma(\mu)}(b_0)_{\Kb_p}\rightarrow \sS_{K}(\Fpbar)$ be the map induced by Axiom \ref{lift}. We define a subset $\sI_{z}$ of $\sS_{K}(\Fpbar)$ as the image of the map
\begin{equation*}
\theta_{K,(z,j)}\colon X_{\sigma(\mu)}(b_0)_{\Kb_p}\times \bG(\A_f^p)\rightarrow \sS_{K}(\Fpbar)
\end{equation*}
induced by $\theta_{z}$ and the prime-to-$p$ Hecke action on $\sS_{K_p}:=\varprojlim_{K^p}\sS_{K^pK_p}$. 

\begin{lem}\label{isps}
\emph{Let $\widetilde{z}\in \Sh_{K}(\bG,\bX)$, and take a representative $(h,g_{0,p},h_{0}^{p})\in \bX \times \bG(\Qp)\times \bG(\A_f^p)$ of $\widetilde{z}$. Assume that the reduction $z\in \sS_{K}(\Fpbar)$ of $\widetilde{z}$ lies in $\sS_{K,[b_0]}$. On the other hand, let $\widetilde{z}'\in \Sh_{K}(\bG,\bX)$ be the point attached to $(h,1)\in \bX\times \bG(\A_f)$, and write $z'$ for the reduction of $\widetilde{z}'$. Then $z'$ lies in $\sI_{z}$. }
\end{lem}

\begin{proof}
By the proof of Proposition \ref{lfok}, that is, the proof of \cite[Proposition 7.8]{Zhou2020} and \cite[Remark 5.4]{Hamacher2019}, there is a map $G(\Qp)\rightarrow X_{\sigma(\mu)}(b)_{\Kb_p}(\Fpbar)$ such that the following diagram commutes:  
\begin{equation*}
\xymatrix@C=46pt{
\{h\} \times G(\Qp)\times \bG(\A_f^p) \ar[r] \ar[d]^{r_h\times [g\mapsto g_0^{-1}g]}& \Sh_{K}(\bG,\bX)(\C) \ar[d] \\
X_{\sigma(\mu)}(b)_{\Kb_p}(\Fpbar)\times \bG(\A_f^p) \ar[r]^{\quad\quad \theta_{K,(z,j)}}& \sS_{K}(\bG,\bX)(\Fpbar). }
\end{equation*}
Hence $z'$ lies in the image of $\theta_{K,(z,j)}$, which is equivalent to $z'\in \sI_{z}$ by definition. 
\end{proof}

The following is a generalization of \cite[Theorem 9.4]{Zhou2020} in the basic case. 

\begin{thm}\label{splf}
\emph{Let $\bT$ be a maximal torus of $\bI_{z}$ which is elliptic over $\Qp$, and $\mu_0 \in X_{*}(\bT_{\Qp})$ satisfying $\mu_0\in \{\mu\}$ and $\overline{\mu}=\nu_{b_0}^{-1}$. Then there is $z'\in \sI_{z}$ and a special point $\widetilde{z}'\in \Sh_{K}(\bG,\bX)(L)$ for a subfield $L$ of $\Qpbar$ which lifts to $z'$ such that the filtration on $\D(X_{K,z})_{\Q}\cong \D(X_{K,z'})_{\Q}$ corresponding to $\widetilde{z}'$ is determined by $\mu_0$. In particular, $\bT \hookrightarrow \Aut_{\Q}\cA_{K,z'}$ lifts to $\bT \hookrightarrow \Aut_{\Q}\cA_{K,\widetilde{z}'}$, and the Mumford--Tate group of $\cA_{K,\widetilde{z}'}$ is contained in $\bT$. }
\end{thm}

\begin{rem}
The existence of $\bT$ and $\mu_0$ in Proposition \ref{splf} are consequences of Lemma \ref{exmt} (ii) and Lemma \ref{adsp} respectively. 
\end{rem}

\begin{proof}
By \cite[\S 10]{Kottwitz1986}, there is a closed immersion $\bT_{\Qp}\hookrightarrow G$ over $\Qp$ which is $G(\Qpbar)$-conjugate to 
\begin{equation*}
\bT_{\Qpbar}\subset \bI_{z,\Qpbar}\xrightarrow{\iota_{p,z}} J_{b_0,\Qpbar}\cong G_{\Qpbar}. 
\end{equation*}
Hence we may assume $b_0\in T(\Qpb)$ by Lemma \ref{adsp}. Then the assertion follows from the same argument as the proof of \cite[Theorem 9.4]{Zhou2020}. 
\end{proof}

\begin{cor}\label{lcif}
\emph{
\begin{enumerate}
\item For a prime $\ell \neq p$, then the morphisms 
\begin{equation*}
\bI_{z}\otimes_{\Q}\Ql \xrightarrow{\iota_{\ell,z}} I_{\ell,z}\xrightarrow{c_{\ell,z}} \bG\otimes_{\Q}\Ql
\end{equation*}
are isomorphisms. 
\item The morphism $\bI_{z}\otimes_{\Q}\Qp \xrightarrow{\iota_{p,z}} J_{b_0}$ is an isomorphism. 
\end{enumerate}}
\end{cor}

\begin{proof}
The proof is the same as that of \cite[Corollary 2.3.2]{Kisin2017} by using Lemma \ref{pwct} for $z_1=z_2=z$. 
\end{proof}

Under the notations in Theorem \ref{splf}, we obtain the following: 
\begin{itemize}
\item the homomorphism $h_0\colon \bS \rightarrow \bT_{\R}$ corresponding to $\mu_0 \colon \G_{m,\C}\rightarrow \bT_{\C}$,
\item a $\bG(\Q)$-conjugacy class $\bc_{\widetilde{z}'}$ of a closed immersion $\bT \hookrightarrow \bG$.
\end{itemize}
Moreover, choosing a morphism in $\bc_{\widetilde{z}'}$ corresponds to choosing a representative of $\widetilde{z}'$ in $\bX \times \bG(\A_f^p)$. 

\begin{cor}\label{indq}
\emph{
There is an inner twisting $\psi \colon \bI_{z}\otimes_{\Q}\Qbar \xrightarrow{\cong} \bG\otimes_{\Q}\Qbar$ such that the following diagram commutes up to inner automorphism for any prime $\ell$: 
\begin{equation*}
\xymatrix@C=46pt{
\bI_{z}\otimes_{\Ql}\Qlbar \ar[r]^{\psi} \ar[d]& \bG \otimes_{\Q}\Qlbar \ar@{=}[d] \\
I_{\ell,z}\otimes_{\Q}\Qlbar \ar[r]^{\cong}& \bG \otimes_{\Q}\Qlbar. }
\end{equation*}
Here $I_{p,z}:=J_{b}$, and the lower horizontal isomorphism is induced by that in Corollary \ref{lcif}. Moreover, for a maximal torus $\bT$ of $\bI_{z}$, we can take $\psi$ as that $\bT_{\Qbar} \subset \bI_{z,\Qbar}\cong \bG_{\Qbar}$ is $\bG(\Qbar)$-conjugate to some (or any) element of $\bc_{\widetilde{z}'}$. }
\end{cor}

\begin{proof}
Take $\bT$ and $\widetilde{z}'$ as in Proposition \ref{splf}. We may assume that $\widetilde{z}'$ admits a representative of the form $(h_0,g_p,1)\in \bX\times \bG(\Qp)\times \bG(\A_f^p)$. Then the assertion follows from the same argument as \cite[Corollary (2.3.5)]{Kisin2017}. 
\end{proof}

We construct a morphism in Theorem \ref{paut}. As in the proof of \cite[Theorem 3.3.2]{Howard2017}, the map $\Theta_{K,(z,j)} \colon \RZ_{K_p}(G,\mu,b_0) \rightarrow \sS_{K}$ and prime-to-$p$ Hecke action on $\sS_{K_p}$ induce a morphism of functors
\begin{equation*}
\bI_{z}(\Q)\backslash \RZ_{K_p}(G,\mu,b_0) \times \bG(\A_{f}^{p})/K^p\rightarrow (\sShat_{K,O_{\Eb}})_{/\sSbar_{K,[b_0]}}. 
\end{equation*}
We also denote it by ${\Theta}_{K,(z,j)}$. Note that this functor induces the map ${\theta}_{K,(z,j)}$ on the $\Fpbar$-valued points. 

\begin{prop}
\emph{
\begin{enumerate}
\item The functor $\bI_{z}(\Q)\backslash \RZ_{K_p}(G,\mu,b_0) \times \bG(\A_{f}^{p})/K^p$ is representable by a formal scheme, which is locally formally of finite type over $\spf O_{\Eb}$. 
\item The morphism ${\Theta}_{K,(z,j)}$ is a monomorphism of functors on $\nilp_{O_{\Eb}}$. 
\end{enumerate}}
\end{prop}

\begin{proof}
These follow from the same argument as \cite[Proposition 4.5]{Kim2018b}. This is pointed out by Wansu Kim. 
\end{proof}

The remaining assertion of Theorem \ref{paut} is the surjectivity of ${\Theta}_{K,(z,j)}$ for any $(z,j)$ as in Axiom \ref{lift}. It suffices to prove the following: 

\begin{prop}\label{unsj}
\emph{For $z_1,z_2\in \sS_{K,[b_0]}$, we have $\sI_{z_1}=\sI_{z_2}$. }
\end{prop}

\begin{proof}
We follow the proof of \cite[Proposition (4.4.13)]{Kisin2017}. 

By Corollary \ref{indq}, there is an isomorphism $\bI_{z_1}=\bI_{z_2}$ as $\Q$-groups. We identify $\bI_{z_1}$ and $\bI_{z_2}$ under the above isomorphism, and denote them by $\bI$. Take $z_1,z_2\in \overline{\sS}_{K,[b_0]}(\Fpbar)$, and fix an elliptic maximal torus $\bT$ of $\bI$. By Proposition \ref{splf}, we may assume that there is a quintiple $(h_0,i_1,i_2,g_1,g_2)$, where
\begin{itemize}
\item $h_0\colon \bS \rightarrow \bT_{\R}$ is a homomorphism, 
\item $i_1,i_2\colon \bT \hookrightarrow \bG$ are closed immersions defined over $\Q$,
\item $g_1,g_2\in \bG(\A_f)$,
\end{itemize}
such that the point $\widetilde{z}_i \in \Sh_{K}(\bG,\bX)$ associated to $(i_j \circ h_0,g_j)\in \bX\times \bG(\A_f)$ lifts $z_j$ for $j\in \{1,2\}$. Let $\bZ^{\circ}$ is the identity component of the center of $\bG$. Since $[b]$ is basic, Lemma \ref{pwct} implies that there is $m\in \Zpn$ such that the $p^m$-power Frobenius on $\cA_{K,z_1}$ and $\cA_{K,z_2}$ are given by an element $\gamma \in \bZ^{\circ}(\Q)$. Moreover, the same argument as the proof of \cite[Proposition (4.4.13)]{Kisin2017} gives us that there is $g\in \bG(\Qbar)$ satisfying $i_2=\Ad(g)\circ i_1$, and we have $g\in (\bG/i_1(\bT))(\Q)_{1}$. Here $(\bG/i_1(\bT))(\Q)_{1}$ is the kernel of the composite map
\begin{equation*}
(\bG/i_1(\bT))(\Q)\rightarrow \rH^1(\Q,i_1(\bT))\rightarrow \rH^1(\R,i_1(\bT))\rightarrow \rH^1(\R,K_{\infty}),
\end{equation*}
and $K_{\infty}$ is the centralizer of $i_1\circ h_0\in X$ in $\bG\otimes_{\Q}\R$. 

Let $\beta$ be the image of $g\in (G/i_1(\bT))(\Q)$ under the composite
\begin{equation*}
(\bG/i_1(\bT))(\Q)\rightarrow \rH^1(\Q,\bT)\rightarrow \rH^1(\Q,\bI_{z_0}). 
\end{equation*}
Then \cite[Lemma (4.4.3) (1)]{Kisin2017} implies $\beta=0$. Hence we have $\sI_{z'_2}=\sI_{z'_1}$ by the same argument as the proof of \cite[Proposition (4.4.8)]{Kisin2017}. Here $z'_j \in \sS_{K}(\Fpbar)$ is the reduction of the point of $\Sh_{K}(\bG,\bX)$ attached to $(i_j\circ h_0,1)\in \bX \times \bG(\A_f)$ for $j\in \{1,2\}$. Applying Lemma \ref{isps} to $\sI_{z_1}$ and $\sI_{z_2}$, we obtain the desired equality $\sI_{z_1}=\sI_{z_2}$. 
\end{proof}

\begin{rem}
If $G$ is quasi-split, then Proposition \ref{unsj} can be proved by the same argument as that of \cite[Proposition (4.4.13)]{Kisin2017} by using \cite[Theorem 9.4]{Zhou2020}. This is pointed out by Pol van Hoften. 
\end{rem}

\end{document}